\documentclass{amsart}
\usepackage{color}
\usepackage{amsmath,amssymb,amsthm}

\theoremstyle{plain}
\newtheorem{theorem}{Theorem}[section]
\newtheorem{lemma}[theorem]{Lemma}
\newtheorem{corollary}[theorem]{Corollary}
\newtheorem{proposition}[theorem]{Proposition}

\theoremstyle{definition}

\newtheorem{example}[theorem]{Example}
\newtheorem{remark}[theorem]{Remark}

\newcommand{\la}{\langle}
\newcommand{\ra}{\rangle}

 \newcommand{\al}{\alpha}
  \newcommand{\be}{\beta}
  \newcommand{\de}{\delta}

\def\bC{\hbox{$\mathbb C$}}
\def\bF{\hbox{$\mathbb F$}}

\def\bN{\hbox{$\mathbb N$}}
\def\bZ{\hbox{$\mathbb{Z}$}}
\def\bR{\hbox{$\mathbb R$}}

\def\bb{\hbox{$\boldsymbol{b}$}}

\def\on{\hbox{${}\|_\omega$}}
\def\mon{\hbox{${}\|_{\infty,1/\omega}$}}

\def\tensor{\hbox{$\widehat{\otimes}$}}

\def\ho{\hbox{$\hat\omega$}}

\def\ga{\hbox{$L^1(G)$}}

\def\L1o{\hbox{$L^1(G,\omega)$}}
\def\l1o{\hbox{$\ell^1(G,\omega)$}}
\def\Ld1o{\hbox{$L^1(G\times G, {\omega\times \omega})$}}
\def\ld1o{\hbox{$\ell^1(G\times G, {\omega\times \omega})$}}
\def\LH1o{\hbox{$L^1(G/H, {\ho})$}}

 \def\C0o{\hbox{$C_0(1/\omega)$}}
 \def\Lio{\hbox{$L^\infty\bigl(G,\frac1{\omega}\bigr)$}}
  \def\lio{\hbox{$\ell^\infty\bigl(G,\frac1{\omega}\bigr)$}}

\def\dLio{\hbox{$L^\infty\bigl(G\times G, \frac{1}{\omega\times \omega}\bigr)$}}
\def\liod{\hbox{$\ell^\infty\bigl(G\times G, \frac{1}{\omega\times \omega}\bigr)$}}

\def\Liot{\hbox{$L^\infty\bigl(G\times G\times G, \frac{1}{\omega\times \omega \times \omega}\bigr)$}}

\newcommand{\Ff}{\mathbb{F}}

\newcommand{\id}{\mathrm{id}}

 \renewcommand{\leq}{\leqslant}
\renewcommand{\geq}{\geqslant}

\newcounter{equi1}

\newcommand{\bea}{\begin{eqnarray*}}
\newcommand{\eea}{\end{eqnarray*}}
\newcommand{\beq}{\begin{equation}}
\newcommand{\eeq}{\end{equation}}
\newcommand{\begsta}{\begin{statements}}
\def\endsta{\end{statements}}
\newcommand{\begaeq}{\begin{aequivalenz}}
\def\endaeq{\end{aequivalenz}}

\begin{document}
\title[Beurling algebras]{Non-weakly amenable Beurling algebras}



\author[V. Shepelska]{Varvara Shepelska\dag}
\email{\dag\; shepelska@gmail.com}

\author[Y. Zhang]{ Yong Zhang \ddag}
\email{\ddag\; zhangy@cc.umanitoba.ca}
\thanks{\ddag \; Supported by NSERC Grant 238949}
\address{Department of Mathematics\\
           University of Manitoba\\
           Winnipeg, Manitoba\\
           R3T 2N2 Canada}

\date{November23, 2014}

\begin{abstract}
Weak amenability of a weighted group algebra, or a Beurling algebra,  is a long-standing open problem. The commutative case has been extensively investigated and fully characterized.
We study the non-commutative case. Given a weight function $\omega$ on a locally compact group $G$, we characterize derivations from $L^1(G,\omega)$ into its dual in terms of certain functions. Then we show that for a locally compact IN group $G$, if there is a non-zero continuous group homomorphism $\varphi$: $G\to \bC$ such that $\varphi(x)/\omega(x)\omega(x^{-1})$ is bounded on $G$, then $L^1(G,\omega)$ is not weakly amenable. Some useful criteria that rule out weak amenability of $L^1(G,\omega)$ are established. Using them we show that for many polynomial type weights the weighted Heisenberg  group algebra is not weakly amenable, neither is the weighted $\boldsymbol{ax+b}$ group algebra. We further study weighted quotient group algebra $L^1(G/H,\hat\omega)$, where $\hat\omega$ is the canonical weight on $G/H$ induced by $\omega$. We reveal that the kernel of the canonical homomorphism from $L^1(G,\omega)$ to $L^1(G/H,\hat\omega)$ is complemented. This allows us to obtain some sufficient conditions under which $L^1(G/H,\hat\omega)$ inherits weak amenability of $L^1(G,\omega)$. We study further weak amenability of Beurling algebras of subgroups. In general, weak amenability of a Beurling algebra does not pass to the Beurling algebra of a subgroup. However, in some circumstances this inheritance can happen. We also give an example to show that weak amenability of both $L^1(H,\omega|_H)$ and $L^1(G/H,\hat\omega)$ does not ensure weak amenability of $L^1(G,\omega)$.
\end{abstract}

\maketitle

\section{Introduction}\label{Intro}

Let $G$ be a locally compact group. As usual, we denote the integral of a function $f$ against a fixed left Haar measure by
\[
\int{f(x)dx}.
\]
 The group algebra $\ga$ is the Banach algebra consisting of all Haar integrable functions on $G$ with the convolution product and the $L^1$-norm
 \[
 \|f\|_1 = \int{|f(x)|dx}.
 \]
  Two functions in $\ga$ are regarded as the same if they are equal almost everywhere on $G$ with respect to the Haar measure.

A {\em weight function} on $G$ is a locally bounded positive measurable function $\omega:G\to\mathbb{R}^{+}$ that satisfies the submultiplicative inequality
\[
\omega(xy)\le\omega(x)\omega(y) \quad (x,y\in G).
\]
Given a weight $\omega$ on $G$, consider
$$
L^1(G,\omega)=\left\{f\,:\,f\omega \in \ga\right\}.
$$
Equipped with the norm
$$
\|f\|_\omega :=\int\limits_G |f(x)|\omega(x)\,dx
$$
and the convolution product, $L^1(G,\omega)$ becomes a Banach algebra, called a {\em weighted group algebra} or a {\em Beurling algebra}. The dual space of $\L1o$ may be identified with
\[
\Lio :=\{f: \, f/\omega \in L^\infty(G)\}
\]
whose norm is given by
\[
\|f\mon = \underset{x\in G}{\textrm{ess}\sup}\,\frac{|f(x)|}{\omega(x)} \quad (f\in \Lio).
\]
Obviously, as a Banach space $\L1o$ is isometrically isomorphic to $\ga$. However, as Banach algebras these two are very different. For example, it is well-known that $\ga$ is a typical quantum group algebra \cite{K-V}, while $\L1o$ is usually not, although $\Lio$ is a von Neumann algebra with the product $f\cdot g = \frac{1}{\omega}fg$. In fact, $\L1o$ is not even an F-algebra, unless the weight is trivial (meaning that the weight is multiplicative). We refer to~\cite{L-Z3} for the relation between quantum groups and F-algebras.

The investigation of $\L1o$ goes back to A. Beurling \cite{Beurling1938}, where $G=\mathbb{R}$ was considered. One may find a good account of elementary theory concerning the general weighted group algebra in \cite{RS}.

Two weight functions $\omega_1$ and $\omega_2$ on $G$ are called {\em equivalent} if there are constants $c_1, c_2 >0$ such that
\[
c_1 \omega_1(x) \leq \omega_2(x) \leq c_2 \omega_1(x)
\]
locally almost everywhere on $G$. It is readily seen that if $\omega_1$ and $\omega_2$ are equivalent weights, then $L^1(G, \omega_1)$ and $L^1(G, \omega_2)$ are isomorphic as Banach algebras. It is well-known that a weight on $G$ is always equivalent to a continuous weight on $G$ (see \cite{White}, or \cite[Theorem~3.7.5]{RS} for a proof; note that in \cite{RS} the condition $\omega \geq 1$ is not necessary if we do not require the weighted algebra to be a subalgebra of~$\ga$). For this reason, unless otherwise is specified, in this paper we always assume that a weight is continuous.

We are concerned with weak amenability of the Beurling algebra $\L1o$. We refer to \cite{D-L-M,F-G-L-L-M,Neufang} for research of other aspects regarding Beurling algebras. Special types of groups have been studied in \cite{H-K-K,L-M-P,Willis}. Related research concerning weighted Fourier algebras may be found in \cite{L-S,L-S-T}.

Recall that a {\em derivation} from a Banach algebra $A$ to a Banach $A$-bimodule $X$ is a linear mapping $D$: $A\to X$ satisfying $D(ab)=a\cdot D(b)+D(a)\cdot b$ ($a,b\in A$). For every $x\in X$ the map $a\mapsto a\cdot x-x\cdot a$ is a continuous derivation, called an {\em inner derivation}. Given a Banach $A$-bimodule $X$, its dual space $X^*$ is naturally a Banach $A$-bimodule (called the dual module of $X$) with the module actions defined by
\[
\la x, a\cdot f\ra = \la x\cdot a, f\ra, \quad \la x, f\cdot a\ra = \la a\cdot x, f\ra \quad (a\in A, f\in X^*, x\in X).
\]
 Following B. E. Johnson \cite{Johnson}, we call $A$ {\em amenable} if every continuous derivation from $A$ into any dual Banach $A$-bimodule $X^*$ is inner. Johnson showed in \cite{Johnson} that the group algebra $\ga$ is amenable if and only if $G$ is an amenable group. Later N. Gronbaek showed in \cite{Gronbaek} that the weighted group algebra $\L1o$ is amenable if and only if $G$ is an amenable group and $\omega$ is a diagonally bounded weight, i.e., the function $\omega(x)\omega(x^{-1})$ is bounded on $G$. The latter conditions actually imply that the weight $\omega$ is bounded up to a multiplicative factor. Hence, a nontrivial weighted group algebra is intrinsically not an amenable Banach algebra.

Weak amenability for commutative Banach algebras was introduced by Bade, Curtis, and Dales in \cite{BCD}. Based on a characterization result of \cite{BCD}, Johnson later called a general Banach algebra $A$ {\em weakly amenable} if every continuous derivation from $A$ into $A^*$ is inner. He showed in \cite{Johnson2} that $\ga$ is weakly amenable for all locally compact groups $G$.

Weak amenability of Beurling algebras has been studied by many authors. In \cite{BCD} it was shown that $L^1(\bZ, \omega_\al)$ for the additive group $\bZ$ and the polynomial weight $\omega_\al(x)= (1+|x|)^\al$ is weakly amenable if and only if $0\leq \al < 1/2$. The same conclusion holds if $\bZ$ is replaced with $\bR$ (\cite{DL, Samei, zhang}). In \cite{Gronbaek2} N. Gronbaek showed that the Beurling algebra of a commutative discrete group $G$ is weakly amenable if and only if every non-trivial group homomorphism $\Phi$: $G\to \bC$ satisfies
\beq\label{homo}
\sup\limits_{g\in G} \frac{|\Phi(g)|}{\omega(g)\omega(g^{-1})}=\infty.
\eeq

It turns out that this characterization is still valid for a general commutative locally compact group.
\begin{theorem}\cite[Theorem~3.1]{zhang}\label{Thm_zhang}
Let $G$ be an Abelian locally compact group, and $\omega$ be a weight on $G$. The Beurling algebra $L^1(G,\omega)$ is weakly amenable if and only if (\ref{homo}) holds for every continuous non-zero group homomorphism $\Phi:G\to\mathbb{C}$.
\end{theorem}

However, condition (\ref{homo}) is far from being sufficient for $\L1o$ to be weakly amenable if the group $G$ is not commutative. A counterexample associated to discrete $SL_2(\mathbb{R})$ was obtained in \cite{Borwick}. In \cite{Shepelska} the first author showed that with a non-trivial polynomial weight $\omega_\al$ the algebra $\ell^1(\bF_2,\omega_\al)$ is never weakly amenable. This contrasts with the results on commutative groups $\bZ$ and $\bR$ mentioned above. Similar investigations concerning the discrete $\boldsymbol{ax+b}$ group were also conducted there. Overall, weak amenability of a non-commutative Beurling algebra is still very unclear. So far we have not even seen a non-trivial example of weakly amenable Beurling algebra which is not commutative. The related problem of weak amenability of the center algebra of a Beurling algebra has been studied in \cite{A-S-S,SZ2,zhang}.

In this paper, in Section~\ref{necessary} we first characterize continuous derivations from $\L1o$ into its dual in terms of certain functions from $\dLio$. We then show that the necessity part of Theorem~\ref{Thm_zhang} remains true if $G$ is an IN group, improving a result of \cite{zhang}. We further establish a criterion that rules out weak amenability of a Beurling algebra. As an application, we show that the weighted group algebra of the topological  Heisenberg group with certain type of ``polynomial weights'' is not weakly amenable.

 In Section~\ref{affine} we continue the investigation of \cite{Shepelska} on weighted $\boldsymbol{ax+b}$ group algebras. For the topological $\boldsymbol{ax+b}$ group, we show that the Beurling algebra on $\boldsymbol{ax+b}$ with a polynomial weight is never weakly amenable. For the discrete case we show that if the weight is independent of $b$, then the corresponding Beurling algebra is weakly amenable only when the weight is diagonally bounded.
 This provides us with an example of a locally compact group $G$ with a closed normal subgroup $H$ and a weight $\omega$ such that both Beurling algebras $L^1(H,\omega|_H)$ and $L^1(G/H,\hat{\omega})$ are weakly amenable, but $L^1(G,\omega)$ is not weakly amenable, where $\hat\omega$ is a weight on $G/H$ naturally induced from $\omega$.

In Section~\ref{c5} we study Beurling algebras associated to quotient groups. If $H$ is a closed normal subgroup of $G$ then
$$
L^1(G/H,\hat{\omega})\cong L^1(G,\omega)/J_{\omega}(G,H),
$$
 where $J_{\omega}(G,H)$ is a closed ideal of $L^1(G,\omega)$. We show that $J_\omega(G,H)$ is always complemented in $\L1o$. This allows us to establish a sufficient condition under which weak amenability of $L^1(G,\omega)$ is inherited by $L^1(G/H,\hat{\omega})$. Using this result, we prove that weak amenability of the tensor product $L^1(G_1,\omega_1)\hat{\otimes}L^1(G_2,\omega_2)$ implies weak amenability of both $L^1(G_1,\omega_1)$ and $L^1(G_2,\omega_2)$, provided the weights $\omega_1$, $\omega_2$ are bounded away from zero.
 The question whether the converse is true remains open except for the case when $G$ is Abelian \cite[Corollary~3.10]{zhang}. We also improve a result of \cite{LL} concerning weak amenability of a complemented subalgebra.

 In Section~\ref{commu} we investigate Beurling algebras of subgroups. Example~\ref{abelian_sum} shows that, even in the Abelian case, weak amenability of a Beurling algebra does not imply weak amenability of the induced Beurling algebra of a subgroup. However, the implication is true under some circumstances. We also investigate the problem of extending a group homomorphism from a subgroup to the whole group in Section~\ref{commu}.

\section{Criteria ruling out weak amenability of $L^1(G,\omega)$}\label{necessary}

We start from a characterization of a bounded derivation from $L^1(G, \omega)$ into its dual $L^1(G, \omega)^* = L^{\infty}\bigl(G,\frac{1}{\omega}\bigr)$. It generalizes a result of B. E. Johnson \cite{Johnson2} which deals with the case $\omega\equiv1$.

Let $G_1$ and $G_2$ be two locally compact groups and $\omega_i$ be a weight on $G_i$ ($i=1,2$). We denote by $\omega_1\times\omega_2$ the weight on $G_1\times G_2$ defined by
$$
(\omega_1\times\omega_2)(x_1,x_2)=\omega_1(x_1)\omega_2(x_2)\quad (x_1\in G_1,\ x_2\in G_2).
$$

\begin{lemma}\label{derivationform}
Let $G$ be a locally compact group and $\omega$ be a weight on $G$. Then for every bounded derivation $D: L^1(G,\omega)\to L^{\infty}\bigl(G,\frac{1}{\omega}\bigr)$ there exists a function $\alpha\in L^{\infty}\bigl(G\times G,\frac{1}{\omega\times\omega}\bigr)$ such that
\begin{equation}\label{12}
\alpha(xy,z)=\alpha(x,yz)+\alpha(y,zx)\ \ \ (\text{locally a.e.}\ \ (x,y,z)\in G\times G\times G)\quad\text{and}
\end{equation}
\begin{equation}\label{13}
\langle g,D(f)\rangle  =\iint\limits_{G\times G} \alpha(x,y)f(x)g(y)\,dxdy\quad (f,g\in L^1(G,\omega)).
\end{equation}
Conversely, every function $\alpha\in L^{\infty}\bigl(G\times G,\frac{1}{\omega\times\omega}\bigr)$ satisfying (\ref{12}) defines a bounded derivation $D: L^1(G,\omega)\to L^{\infty}(G,1/\omega)$ by the formula (\ref{13}).
\end{lemma}

\begin{proof}
Given a bounded derivation $D: L^1(G,\omega)\to L^{\infty}\bigl(G,\frac{1}{\omega}\bigr)$, the map $(f,g)\mapsto \langle g,D(f)\rangle$ is a bilinear functional on $\L1o\times \L1o$ and we have
\[
|\la g, D(f)\ra| \leq \|D\|\,\|g\on \|f\on.
\]
Hence this map defines a bounded linear functional
\[
\al \in \left(\L1o\tensor\L1o\right)^* = L^{\infty}\Bigl(G\times G,\frac{1}{\omega\times\omega}\Bigr)
\]
by
\[
\langle f\otimes g, \alpha\rangle = \langle g,D(f)\rangle \quad (f,g\in L^1(G,\omega)).
\]
It follows that relation (\ref{13}) holds.

Let $\pi$: $\Lio \to \dLio$ be the operator defined by
\[
\pi(f)(x,y) = f(xy) \quad \left(f\in L^{\infty}\Bigl(G,\frac{1}{\omega}\Bigr)\right).
\]
From \cite[Corollary~3.3.32]{RS} it is readily seen that $\pi(f)\in \dLio$ if $f\in \Lio$, and $\|\pi(f)\|_{\infty,1/(\omega\times\omega)} = \|f\mon$. 

Applying $\pi\otimes \id$ to $\al$, where $\id$ stands for the identity operator on $\Lio$, we see that the function $\al_1(x,y,z) = \al(xy,z)$ belongs to $\Liot$. Similarly, the functions  $\al_2(x,y,z)=\al(x,yz)$ and $\al_3(x,y,z) = \al(y,zx)$ also belong to $\Liot$. In order to show that identity (\ref{12}) holds, it suffices to verify
\[
\la f\otimes g\otimes h, \al_1\ra = \la f\otimes g\otimes h, \al_2\ra + \la f\otimes g\otimes h, \al_3\ra .
\]
In fact,
\begin{align*}
\la f\otimes g\otimes h, \al_1\ra &=\int\limits_{G^3}{\al(xy,z)f(x)g(y)h(z)}dxdydz\\
                                                    & = \int\limits_{G^2}{\al(y,z)(f*g)(y)h(z)}dydz = \la h, D(f*g)\ra\\
           &= \la h, f\cdot D(g) + D(f)\cdot g\ra  = \la h*f, D(g) \ra + \la g*h, D(f) \ra \\
         &=  \la f\otimes g\otimes h, \al_2\ra + \la f\otimes g\otimes h, \al_3\ra
\end{align*}
for all $f,g,h\in \L1o$. Therefore (\ref{12}) holds.

The converse can be easily verified by computation. The proof is complete.

\end{proof}

Recall that a locally compact group $G$ is an IN group if it has a compact neighborhood of the unit element $e$ which is invariant under all inner automorphisms, i.e., if there is a compact neighborhood $U$ of $e$ such that $gUg^{-1} = U$ for all $g\in G$. It was shown in \cite[Remark~3.2]{zhang} by the second author that for an IN group $G$ the necessity part of Theorem~\ref{Thm_zhang} remains true under some extra condition. We now can remove this condition. Precisely, we have the following.

\begin{theorem}\label{IN}
Let $G$ be an IN group and $\omega$ be a weight on $G$. Suppose that there exists a non-trivial continuous group homomorphism $\Phi: G\to\mathbb{C}$ such that
\[
\sup\limits_{t\in\,G}\,\frac{|\Phi(t)|}{\omega(t)\omega(t^{-1})}<\infty.
\]
Then $L^1(G,\omega)$ is not weakly amenable.
\end{theorem}

\begin{proof}
We use $\Phi$ to construct a continuous non-inner derivation $D:L^1(G,\omega)\to L^{\infty}\bigl(G,\frac1{\omega}\bigr)$. Let $B$ be an invariant compact neighborhood of $e$. Define $D$ as in \cite[Theorem~3.1]{zhang} by
\begin{equation}\label{62}
D(h)(t)=\int\limits_{B} \Phi(t^{-1}\xi)h(t^{-1}\xi)\,d\xi\quad (h\in L^1(G,\omega),\ t\in G).
\end{equation}
As indicated in \cite[Remark~3.2]{zhang}, $D$ is indeed a continuous derivation. Here we show this by using Lemma~\ref{derivationform}. For all $g,h\in L^1(G,\omega)$ we have
\[
\langle g, D(h)\rangle =\int\limits_G \int\limits_{t^{-1}B} \Phi(\xi)h(\xi)\,d\xi\,g(t)\,dt=\int\limits_G\int\limits_G \chi_{\phantom{}_{t^{-1}B}}(\xi)\Phi(\xi)h(\xi)g(t)\,d\xi dt.
\]
Let $\alpha(\xi,t)= \chi_{\phantom{}_{B}}(t\xi)\Phi(\xi)$. Then $\alpha$ is clearly a measurable function on $G\times G$. Also,
\begin{align*}
\sup\limits_{(\xi,t)\in G\times G}\,\frac{|\alpha(\xi,t)|}{\omega(\xi)\omega(t)} &=\sup\limits_{\xi,t\in G}\,\frac{|\chi_{\phantom{}_{B}}(t\xi)\Phi(\xi)|}{\omega(\xi)\omega(t)}=\sup\limits_{\xi,t\in G,\,t\xi\in B}\,\frac{|\Phi(\xi)|}{\omega(\xi)\omega(t)}\\
&\le \sup\limits_{\xi\in G}\, \frac{|\Phi(\xi)|}{\omega(\xi)\omega(\xi^{-1})}\cdot \sup\limits_{\xi,t\in G,\,t\xi\in B}\,\frac{\omega(\xi^{-1})}{\omega(t)}\\
&\le \sup\limits_{\xi\in G}\, \frac{|\Phi(\xi)|}{\omega(\xi)\omega(\xi^{-1})}\cdot \sup\limits_{\xi,t\in G,\,t\xi\in B}\,\omega((t\xi)^{-1})<\infty,
\end{align*}
since $\sup_{\xi\in G}\, \frac{|\Phi(\xi)|}{\omega(\xi)\omega(\xi^{-1})}<\infty$ and $\omega$ is bounded on the compact set $B^{-1}$ as a continuous function. So we have shown that $\alpha\in L^{\infty}\bigl(G\times G,\frac1{\omega\times\omega}\bigr)$. Next we prove that
\begin{equation}\label{70}
\alpha(xy,z)=\alpha(x,yz)+\alpha(y,zx)\quad (x,y,z\in G).
\end{equation}
Fix $x,y,z\in G$. Since $yzx=y(zxy)y^{-1}$ and $B$ is invariant under inner automorphisms, we have that $\chi_{\phantom{}_B}(zxy)=\chi_{\phantom{}_B}(yzx)$. Then we can use the fact that $\Phi$ is a group homomorphism to obtain
\begin{align*}
\alpha(xy,z)&=\chi_{\phantom{}_B}(zxy)\Phi(xy)=\chi_{\phantom{}_B}(zxy)(\Phi(x)+\Phi(y))=\chi_{\phantom{}_B}(yzx)\Phi(x)+ \chi_{\phantom{}_B}(zxy)\Phi(y)\\&=\alpha(x,yz)+\alpha(y,zx).
\end{align*}
So identity (\ref{70}) is verified. By Lemma~\ref{derivationform}, $D$ is a bounded derivation from $L^1(G,\omega)$ to $L^{\infty}\bigl(G,\frac1{\omega}\bigr)$.

We now show that for every $h\in L^1(G,\omega)$ the function $D(h)\in L^{\infty}(G,1/\omega)$ is continuous. Fix any $t_0\in G$ and let $C$ be a compact neighborhood of $t_0$. Let
$$
\beta(x)=\begin{cases}\Phi(x)h(x),&\,x\in C^{-1}B,\\ 0,&\,x\notin C^{-1}B.\end{cases}
$$
Then,
\[
 D(h)(t)=\int\limits_B\, \beta(t^{-1}\xi)\,d\xi = \int\limits_B{L_t(\beta)(\xi)}d\xi \quad (t\in C),
\]
where $L_t$ is the left translation operator.
Since $\Phi$ is continuous, $C^{-1}B$ is compact, $h\in L^1(G,\omega)$, and $\omega$ is bounded away from zero on compact sets, we have that $\beta\in L^1(G)$. Therefore, for $t\in C$ we have:
\begin{align*}
|D(h)(t)-D(h)(t_0)|&=\left|\int\limits_{B} \bigl(L_{t}\beta(\xi)- L_{t_0}\beta(\xi)\bigr)\,d\xi\right|\le \int\limits_{G} |L_{t}\beta(\xi)- L_{t_0}\beta(\xi)|\,d\xi\\
&=\|L_{t}\beta- L_{t_0}\beta\|_{L^1(G)}\to0\quad \text{as }t\to t_0.
\end{align*}
Hence, $D(h)$ is continuous at $t_0$. Since $t_0$ was taken arbitrarily, we conclude that $D(h)$ is a continuous function on $G$  for each $h\in L^1(G,\omega)$.

We are now ready to show that $D$ is not an inner derivation. Suppose, to the contrary, that there exists $f\in L^{\infty}\bigl(G,\frac1{\omega}\bigr)$ such that
\begin{equation}\label{63}
D(h)=f\cdot h-h\cdot f\quad (h\in L^1(G,\omega)).
\end{equation}
Fix any $t_0\in G$ and consider $h_0=\chi_{\phantom{}_{t_0^{-1}B}}\in L^1(G,\omega)$. Then
\begin{align*}
&D(h_0)(t_0)=(f\cdot h_0)(t_0)-(h_0\cdot f)(t_0)=\int\limits_{G} f(yt_0)h_0(y)\,dy-\int\limits_{G} f(t_0y)h_0(y)\,dy\\
&=\ \int\limits_{t_0^{-1}B} f(yt_0)\,dy-\int\limits_{t_0^{-1}B} f(t_0y)\,dy=
\int\limits_{t_0^{-1}Bt_0} f(y)\,dy-\int\limits_{B} f(y)\,dy=0.
\end{align*}
As we have already shown, $D(h_0)$ is a continuous function. It is also standard that $f\cdot h_0-h_0\cdot f$ is a continuous function when $f\in L^{\infty}\bigl(G,\frac1{\omega}\bigr)$ (see, for example, \cite[Proposition~7.17]{DL}). Therefore,
$$
0=D(h_0)(t_0)=\int\limits_B \Phi(t_0^{-1}\xi)h_0(t_0^{-1}\xi)\,d\xi=\int\limits_{B} \Phi(t_0^{-1}\xi)\,d\xi.
$$
Since $\Phi$ is a homomorphism, we obtain
$$
0=\int\limits_{B} \Phi(t_0^{-1}\xi)\,d\xi=\int\limits_B (\Phi(\xi)-\Phi(t_0))\,d\xi=\int\limits_B \Phi(\xi)\,d\xi-\Phi(t_0)\mu(B),$$
which implies that
$$
\Phi(t_0)=\frac{\int_B \Phi(\xi)\,d\xi}{\mu(B)},
$$
where $\mu$ denotes the Haar measure on $G$ ($\mu(B)>0$ since $B$ is a neighborhood of identity and thus contains an open subset). Because $t_0\in G$ was chosen arbitrarily, it follows that $\Phi$ is constant on $G$, which can happen for a homomorphism $\Phi$ only if $\Phi\equiv0$. This contradiction shows that $D$ is not an inner derivation. The proof is complete.

\end{proof}


Our next result provides another criterion to rule out weak amenability for a Beurling algebra. For the discrete case it was first obtained by Borwick in his Ph.D.~thesis \cite{Borwick} (see also \cite{Shepelska}), and has been used in \cite{Borwick} and \cite{Shepelska} to study weak amenability of Beurling algebras on discrete $SL_2(\bR)$, $\Ff_2$, and discrete $\boldsymbol{ax+b}$ group.

Let $G$ be a group. Recall that the conjugacy class of $x\in G$ is the set $C_x =\{gxg^{-1}:\, g\in G\}$. Given a subset $B$ of $G$, we denote
\[
C_B=\{gxg^{-1}: g\in G,\,x\in B\} = \bigcup\limits_{x\in B}C_x
\]
and call it the conjugacy class of $B$.

\begin{theorem}\label{conjugacygeneral}
Let $G$ be a locally compact group, $B\ne\emptyset$ be an open set in $G$ with compact closure, and $\omega$ be a weight on $G$ that is bounded away from zero on $C_B$, i.e., there is a constant $\delta>0$ such $\omega(x)\ge\delta$ for $x\in C_B$. Suppose that there exists a measurable function $\psi:G\to\mathbb{C}$ bounded on $B$ and such that
\begin{equation}\label{6}
\underset{x,y\in\,G}{\mathrm{ess\,sup}}\,\frac{|\psi(xy)-\psi(yx)|}{\omega(x)\omega(y)}<\infty\quad\text{and}
\end{equation}
\begin{equation}\label{7}
\underset{z\in\,C_B}{\mathrm{ess\,sup}}\,\frac{|\psi(z)|}{\omega(z)}=\infty.
\end{equation}
Then $L^1(G,\omega)$ is not weakly amenable.
\end{theorem}

\begin{proof}
Suppose that $\psi$ is a function satisfying all aforementioned conditions. Then $\Psi(x,y)=\psi(xy)-\psi(yx)$ is measurable on $G\times G$, and condition~(\ref{6}) ensures that $\Psi\in\dLio$. Moreover,
\begin{align*}
\Psi(xy,z)&=\psi(xyz)-\psi(zxy)=(\psi(xyz)-\psi(yzx))+(\psi(yzx)-\psi(zxy))\\
          &=\Psi(x,yz)+\Psi(y,zx) \qquad (x,y,z \in G).
\end{align*}
Then by Lemma~\ref{derivationform} $\Psi$ defines a continuous derivation $D: L^1(G,\omega)\to \Lio$ that satisfies
$$
\langle g,D(f)\rangle  =\int\limits_{G^2} (\psi(xy)-\psi(yx)) f(x) g(y)\,dxdy \quad (f,g\in L^1(G,\omega)).
$$
We show that this derivation $D$ is not inner, which will imply that $\L1o$ is not weakly amenable.

Suppose, to the contrary, that $D$ is inner. Then there exists a function $\varphi\in \Lio$ such that
$$
D(f)=\varphi\cdot f-f\cdot\varphi\quad (f\in L^1(G,\omega)).
$$
It follows that
\[
\langle g,D(f)\rangle =\int\limits_{G^2} (\varphi(xy)-\varphi(yx)) f(x)g(y)\,dxdy \quad (f,g\in L^1(G,\omega)).
\]
Denote $\Phi(x,y) = \varphi(xy)-\varphi(yx)$. Then $\Phi \in \dLio$ and
\[
\la f\otimes g, \Psi-\Phi \ra = 0 \quad (f,g\in \L1o).
\]
Therefore, $\Psi = \Phi$ as the elements of $\dLio$. We then have
\[
\int\limits_{G^2} (\Psi(x,y)-\Phi(x,y)) U(x,y)\,dxdy =0\quad (U\in \Ld1o).
\]
On the other hand, if $U$ is in $\Ld1o$, then so is the function $\chi_{\phantom{}_{B}}(xy)U(x,y)$. Hence,
\[
\int\limits_{G^2} (\Psi(x,y)-\Phi(x,y))\chi_{\phantom{}_{B}}(xy) U(x,y)\,dxdy =0.
\]
 In particular, the last equality holds for all $U$ in $C_{00}(G\times G)$, the space of continuous functions with compact support. For any $U\in C_{00}(G\times G)$, let $V(x,y) = U(x, xy)$. It is evident that $V\in C_{00}(G\times G)$. Thus,
\begin{align*}
0 &= \int\limits_{G^2} (\Psi(x,y)-\Phi(x,y)) \chi_{\phantom{}_{B}}(xy) V(x,y)\,dxdy\\ &= \int\limits_{G\times B} (\Psi(x,x^{-1}y)-\Phi(x,x^{-1}y)) U(x,y)\,dxdy
\end{align*}
for all $U\in C_{00}(G\times G)$. Since $C_{00}(G\times G)$ is dense in $\Ld1o$, we have
\[
\int\limits_{G\times B} (\Psi(x,x^{-1}y)-\Phi(x,x^{-1}y)) U(x,y)\,dxdy =0 \quad (U\in \Ld1o).
\]
This implies that $\Psi(x,x^{-1}y)-\Phi(x,x^{-1}y) = 0$ locally almost everywhere on $G\times B$, i.e.,
\[
\psi(x^{-1}yx) = \psi(y) -\varphi(y) + \varphi(x^{-1}yx) \quad (\text{locally a.e. on }G\times B).
\]
Dividing both sides by $\omega(x^{-1}yx)$ and noting that
\[
\underset{(x,y)\in G\times G}{\mathrm{ess\,sup}}\,\frac{|\varphi(x^{-1}yx)|}{\omega(x^{-1}yx)} = \|\varphi\mon,
\]
we obtain
\[
\frac{|\psi(x^{-1}yx)|}{\omega(x^{-1}yx)} \leq \frac{\omega(y)\|\varphi\mon + |\psi(y)|}{\omega(x^{-1}yx)} + \|\varphi\mon\quad (\text{locally a.e.}\ (x,y)\in G\times B).
\]
Since $\|\varphi\mon < \infty$, $\psi$ and $\omega$ are bounded on $B$, and $\omega$ is bounded away from zero on $C_B$, we derive
\[
\underset{(x,y)\in G\times B}{\mathrm{ess\,sup}}\, \frac{|\psi(x^{-1}yx)|}{\omega(x^{-1}yx)} < \infty,
\]
which is a contradiction to condition~(\ref{7}). Therefore, $D$ is not inner. The proof is complete.

\end{proof}

As an application of Theorem~\ref{conjugacygeneral}, let us consider the topological Heisenberg group. Recall that the Heisenberg group $G_H$ is a 3-dimensional Lie group consisting of all $3\times 3$ matrices of the form
$$
\left ( \begin{array}{ccc}
1 & u & w \\
0 & 1 & v \\
0 & 0 & 1 \end{array} \right )\quad
(u,v,w\in\mathbb{R}).
$$
It is a unimodular locally compact group with the ordinary Euclidean norm topology and the Lebesgue measure of $\mathbb{R}^3$ as a Haar measure (see \cite[Section~12.1.18]{Palmer2}). To simplify the notation, we represent the elements of $G_H$ by $(u,v,w)$ so that $G_H=\mathbb{R}^3$ with the product and inverse operations given by
\begin{equation}\label{58}
(u,v,w)(a,b,c)=(u+a,v+b,w+c+ub),\quad (u,v,w)^{-1}=(-u,-v,uv-w).
\end{equation}

\begin{proposition}\label{Heisenberg}
Let $\omega$ be a weight on $G_H$ of the form

\[
\omega(u,v,w) = W(|u|,|v|) \quad\left((u,v,w)\in G_H\right).
\]
 Suppose that
 \beq\label{unbounded W}
 \lim_{(x,y)\to \infty}W(x,y) = \infty.
 \eeq
Then $L^1(G_H,\omega)$ is not weakly amenable.
\end{proposition}

\begin{proof}
Consider $B=\{(u,v,w): |u|<1,\,|v|<1,\,|w|<1\}$. Then $B$ is an open set in $G_H$ with compact closure. From (\ref{58}) we have
$$
(u,v,w)(a,b,c)(u,v,w)^{-1}=(a,b,c+ub-va).
$$
Therefore $C_B=\{(u,v,w): |u|<1,\,|v|<1,\,w\in\mathbb{R}\}$. Since $\omega>0$ is continuous and depends only on the first two variables, it is obviously both bounded and bounded away from zero on $C_B$. Consider
$$
\tilde{\omega}(t)=\inf\{ W(u,v): u\geq 0, v\geq 0, u+v>|t|\}.
$$
It is readily seen that $\tilde\omega$ is a positive increasing unbounded continuous function on~$\bR$ and  $\tilde\omega(-t) = \tilde\omega(t)$ ($t\in\mathbb{R}$). Moreover, $\tilde\omega$ is a weight on $(\bR,+)$. To see this we note that if $u_i, v_i \geq 0$, $t_i\in \bR$ and $u_i + v_i > |t_i|$ ($i = 1,2$), then
\[
\tilde\omega(t_1 + t_2) \leq W(u_1+u_2, v_1+v_2) \leq W(u_1,v_1) W(u_2, v_2).
\]
Taking infimum on the right side over all possible $(u_1,v_1)$ and $(u_2,v_2)$, we derive the desired inequality
\[
\tilde\omega(t_1 + t_2) \leq \tilde\omega(t_1)\tilde\omega(t_2) \quad (t_1,t_2 \in \bR).
\]
Let
\[
\psi(u,v,w) = \chi_{_{C_B}}(u,v,w)\,\ln \tilde\omega(w) \quad\left((u,v,w)\in G_H\right),
\]
where $\chi_{_{C_B}}$ is the characteristic function of $C_B$. We aim to show that $\psi$ satisfies all the conditions of Theorem~\ref{conjugacygeneral}. It is readily seen that $\psi$ is a locally bounded measurable function on $G_H$ which is unbounded on $C_B$ by (\ref{unbounded W}). Since $\omega$ is bounded on $C_B$, it follows that $\psi$ satisfies condition~(\ref{7}). To show that (\ref{6}) is satisfied, we let $\boldsymbol{x} = (u,v,w)\in G_H$ and $\boldsymbol{y} = (a,b,c)\in G_H$. Then $\boldsymbol{x}\boldsymbol{y}$ and $\boldsymbol{y}\boldsymbol{x}$ belong to the same conjugacy class. If $\boldsymbol{x}\boldsymbol{y}\notin C_B$, then $\boldsymbol{y}\boldsymbol{x}\notin C_B$ and condition~(\ref{6}) is obviously satisfied. Assume now that $\boldsymbol{x}\boldsymbol{y},\boldsymbol{y}\boldsymbol{x}\in C_B$. Then
\begin{equation}\label{20}
|\psi(\boldsymbol{x}\boldsymbol{y}) - \psi(\boldsymbol{y}\boldsymbol{x})| = \left|\ln\frac{\tilde{\omega}(w+c+ub)}{\tilde{\omega}(w+c+av)}\right| \le|\ln\tilde{\omega}(|ub-av|)|=\ln\tilde{\omega}(|ub-av|).
\end{equation}
To obtain the last inequality, we used symmetry and submultiplicativity of $\tilde\omega$ together with the fact that $\tilde\omega\ge1$ as a symmetric weight function.
 Since  $\boldsymbol{x}\boldsymbol{y}\in C_B$, we have that $|u+a|<1$ and $|v+b|<1$. So,
\[
|ub-av|=|(u+a)b-a(v+b)| \leq |a| + |b|.
\]
Similarly, $|ub-av|\leq |u| + |v|$. Then the monotonicity of $\tilde\omega$ implies
\begin{align*}
\ln\tilde{\omega}(|ub-av|) &\leq \frac{1}{2}\ln\left(\tilde{\omega}(|a|+|b|)\, \tilde{\omega}(|u|+|v|)\right) \\
                            &\leq \frac{1}{2}\ln\left(W(|a|,|b|)\, W(|u|,|v|)\right) = \frac{1}{2}\ln \left(\omega(\boldsymbol{x})\omega(\boldsymbol{y})\right)\leq \frac{1}{2}\omega(\boldsymbol{x})\omega(\boldsymbol{y}).
\end{align*}
In the last step we used the fact that $\omega \geq 1$, which is true since $\omega$ is a symmetric weight by the assumption. Combining the last inequality with (\ref{20}), we see that $\psi$ satisfies condition~(\ref{6}). By Theorem~\ref{conjugacygeneral}, $L^1(G_H,\omega)$ is not weakly amenable, and the proof is complete.

\end{proof}

It is readily seen that the function $\omega_{\alpha}(u,v,w)=(1+|u|+|v|)^{\alpha}$ 
is a weight on $G_H$ satisfying the condition of Proposition~\ref{Heisenberg}. So we have
\begin{example}
 The Beurling algebra $L^1(G_H,\omega_{\alpha})$ 
is not weakly amenable for any $\al>0$.
\end{example}

It is worth to restate Theorem~\ref{conjugacygeneral} for the discrete group case. We will use this discrete version to study weak amenability of $\ell^1(\boldsymbol{ax+b},\omega)$ in Section~\ref{affine}.

\begin{corollary}\label{conjugacy}
Let $G$ be a discrete group, $B\ne \emptyset$ be a finite set in $G$, and $\omega$ be a weight on $G$ that is bounded away from zero on the conjugacy class $C_B$. Suppose that there exists a function $\psi:G\to\mathbb{R}$ and a constant $c>0$ such that
\begin{equation}\label{40}
|\psi(xy)-\psi(yx)|\le c\,\omega(x)\omega(y)\quad (x,y\in G)\quad \text{and}
\end{equation}
\begin{equation}\label{39}
\sup_{z\in C_B}\,\frac{|\psi(z)|}{\omega(z)}=\infty.
\end{equation}
Then $\ell^1(G,\omega)$ is not weakly amenable.
\end{corollary}
\qed

For a discrete group $G$, Lemma~\ref{derivationform} ensures that each bounded derivation $D$: $\l1o \to \lio$ gives rise to a function $\al\in \liod$ such that
\[
\al(xy,z) =\al(x,yz) + \al(zx,y) \text{\, and \,}  D(\de_x)(y) = \al(x,y) \quad (x,y,z\in G).
\]
With an additional assumption we can derive further that $D$ must be in the form
\[
D(\de_x) = f\cdot \de_x - \de_x\cdot f, \text{\,\, i.e., \,\,} \al(x,y) = f(xy) - f(yx) \quad (x,y\in G)
\]
for some function $f$ on $G$. We note that although $\al\in \liod$, in general one cannot expect that $f\in \lio$, which happens only when $D$ is an inner derivation.

\begin{lemma}\label{tech2}
Let $G$ be a discrete group, $\omega$ be a weight on $G$, and $D:\ell^1(G,\omega)\to\ell^{\infty}\bigl(G,\frac1{\omega}\bigr)$ be a bounded derivation. If $D(\delta_x)(y)=0$ for all commuting elements $x,y\in G$, then there exists a function $f$ on $G$ such that
\begin{equation}\label{57}
D(\delta_x)(y)=f(xy)-f(yx)\quad (x,y\in G).
\end{equation}
\end{lemma}

\begin{proof}
Since every element commutes with the unit $e$, from our assumption it follows that $D(\de_x)(e) = D(\de_e)(x) = 0$ for all $x\in G$. In particular, $D(xy)(e) =0$, which implies that $D(\de_x)(y) = - D(\de_y)(x)$ for all $x,y\in G$.

We note that $G$ is the disjoint union of all conjugacy classes. To construct $f$ we consider each conjugacy class separately. Let $x_0\in G$ be fixed. Define $f$ on $C_{x_0} =\{yx_0y^{-1}: y\in G\}$ as follows:
\[
f(yx_0y^{-1})=-D(\delta_{x_0y^{-1}})(y)\quad (y\in G).
\]
We first clarify that $f$ is well-defined. Suppose that $u\in C_{x_0}$ has two representations $u=yx_0y^{-1}=zx_0z^{-1}$. Then $x_0y^{-1}=y^{-1}zx_0z^{-1}$. Using the derivation identity, we obtain
\begin{align*}
{D(\delta_{x_0y^{-1}})(y)}&=D(\delta_{(y^{-1}z)(x_0z^{-1})})(y)=\left(D(\delta_{y^{-1}z})\cdot\delta_{x_0z^{-1}} +\delta_{y^{-1}z}\cdot D(\delta_{x_0z^{-1}})\right)(y)\\
                          &=D(\delta_{y^{-1}z})(x_0z^{-1}y)+{D(\delta_{x_0z^{-1}})(z)}.
\end{align*}
Since $yx_0y^{-1}=zx_0z^{-1}$, it is readily seen that the elements $y^{-1}z$ and $x_0z^{-1}y$ commute. By assumption, we then  have $D(\delta_{y^{-1}z})(x_0z^{-1}y) =0$.
Thus,
\[
{D(\delta_{x_0y^{-1}})(y)} = {D(\delta_{x_0z^{-1}})(z)}.
\]
This shows that the function $f$ is well-defined on $C_{x_0}$, so it is well-defined on the whole $G$. (Here, of course, the Axiom of Choice is assumed.) We now prove (\ref{57}). For any $x,y\in G$ the elements $xy$ and $yx$ belong to the same conjugacy class, say~$C_{x_0}$. Let $xy = ax_0a^{-1}$. Then
\begin{align*}
f(xy) &= -D(\de_{x_0a^{-1}})(a) = D(\de_a)(x_oa^{-1}),\\
f(yx) &= f(yax_0(ya)^{-1}) = D(\de_{ya})(x_0a^{-1}y^{-1}) = D(\de_a)(x_0a^{-1}) + D(\de_y)(x).
\end{align*}
In the last step we used the relation $ax_0a^{-1}y^{-1} = x$. Therefore,
\[
f(xy) - f(yx) = -D(\de_y)(x) = D(\de_x)(y).
\]
The proof is complete.

\end{proof}

\begin{proposition}\label{almostsufficient}
Let $G$ be a discrete group and $\omega$ be a weight on $G$ such that
$$
\sup\limits_{n\in\,\mathbb{N}}\,\frac{n}{\omega(x^n)\omega(x^{-n})}=\infty\quad (x\in G).
$$
Then for every bounded derivation $D:\ell^1(G,\omega)\to\lio$ there exists a function $f$ on $G$ such that
\[
D(\delta_x)(y)=f(xy)-f(yx)\quad (x,y\in G).
\]
\end{proposition}

\begin{proof}
Due to Lemma~\ref{tech2}, it suffices to show that $D(\delta_x)(y)=0$ for all bounded derivations $D$: $\l1o \to \lio$ and all commuting elements $x,y\in G$. Suppose, to the contrary, that $xy=yx$ and $D(\delta_x)(y)=c\ne0$ for some bounded derivation $D$. Then, by induction, we have
\begin{equation}\label{29}
D(\delta_{x^{n}})(yx^{1-n})=cn\quad (n\in\mathbb{N}).
\end{equation}
In fact, this is trivial for $n=1$. Now assume that (\ref{29}) holds for $n\in\mathbb{N}$. Then
\begin{align*}
&D(\delta_{x^{n+1}})(yx^{-n})= \left(D(\delta_x)\cdot\delta_{x^n}+\delta_x\cdot D(\delta_{x^n})\right)(yx^{-n})\\
&=D(\delta_x)(y)+D(\delta_{x^n})(yx^{1-n}) = c+cn = c(n+1).
\end{align*}
So (\ref{29}) holds for all $n\in \bN$.
It then follows that
\begin{align*}
\|D\|&\ge \sup_{n\in\,\mathbb{N}}\,\frac{\|D(\delta_{x^n})\|_{\ell^{\infty}(G,1/\omega)}}{\|\delta_{x^{n}}\|_ {\ell^1(G,\omega)}}
\ge\sup_{n\in\,\mathbb{N}}\,\frac{\frac{|D(\delta_{x^n})(yx^{1-n})|}{\omega(yx^{1-n})}}{\omega(x^n)}\\
 & = \sup_{n\in\,\mathbb{N}}\,\frac{|c|n}{\omega((yx)x^{-n})\omega(x^n)}\ge \sup_{n\in\,\mathbb{N}}\,\frac{|c|n}{\omega(yx)\omega(x^{-n})\omega(x^n)}\\
&=\frac{|c|}{\omega(yx)}\,\sup\limits_{n\in\,\mathbb{N}}\,\frac{n}{\omega(x^{-n})\omega(x^n)}=\infty
\end{align*}
due to the condition on $\omega$. This contradicts to the boundedness of $D$. The proof is complete.

\end{proof}

\begin{remark}
Taking into account Lemma~\ref{derivationform}, we see that the function $f$ ensured in Lemma~\ref{tech2} and Proposition~\ref{almostsufficient} satisfies
\[
\sup_{x,y\in G}\frac{|f(xy) - f(yx)|}{\omega(x)\omega(y)} < \infty.
\]
\end{remark}

\section{The affine motion group}\label{affine}

In this section we consider the $\boldsymbol{ax+b}$ group
of all affine transformations $x\mapsto ax+b$ of~$\mathbb{R}$ with $a>0$ and $b\in\mathbb{R}$. Precisely, $\boldsymbol{ax+b} = \{(a,b): \, a\in \bR^+, b\in \bR\}$ with product and inverse given by
$$
(a,b)(c,d)=(ac,ad+b),\quad (a,b)^{-1}=\left(\frac1a,\frac{-b}{a}\right) \quad (a,c \in \bR^+,\,b,d\in\mathbb{R}).
$$
With the Euclidean metric topology inherited from $\bR^2$, $\boldsymbol{ax+b}$ is a locally compact group whose left Haar measure is $da\,db/a^2$. 

Lets consider the function $\omega_{\alpha}(a,b)=(1+a+|b|)^{\alpha}$ on $\boldsymbol{ax+b}$ ($\alpha>0$). For $(a,b),(c,d) \in \boldsymbol{ax+b}$ we have
\begin{align*}
\omega_{\alpha}\bigl((a,b)(c,d)\bigr) &=\omega_{\alpha}(ac,ad+b)=(1+ac+|ad+b|)^{\alpha}\le(1+|b|+ac+a|d|)^{\alpha}\\ &\le (1+a+|b|)^{\alpha}(1+c+|d|)^{\alpha}=\omega_{\alpha}(a,b)\omega_{\alpha}(c,d).
\end{align*}
This shows that $\omega_\al$ is indeed a (continuous) weight on $\boldsymbol{ax+b}$.

\begin{proposition}
Let $\omega_{\alpha}$ ($\alpha>0$) be the weight on $\boldsymbol{ax+b}$ defined as above. Then $L^1(\boldsymbol{ax+b},\omega_{\alpha})$ is not weakly amenable.
\end{proposition}

\begin{proof}
Clearly, $\omega_{\alpha} \ge 1$ on $\boldsymbol{ax+b}$. 
Let $B=\{(a,b): 1<a<2,\,1<b<2\}$. Then $B$ is open and $\overline{B}$ is compact in $\boldsymbol{ax+b}$. Since
$$
(c,d)(a,b)(c,d)^{-1}=(ac,bc+d)\left(\frac 1c,-\frac dc\right)=(a,-ad+bc+d),
$$
we have that $C_B=\{(a,b): 1<a<2,\,b\in\mathbb{R}\}$.
Consider the auxiliary function $\Psi:\boldsymbol{ax+b}\to\mathbb{R}^+$ defined by
$$
\Psi(a,b)=\begin{cases} \max\{a-1,|b|\}&\ \text{if}\ 1<a<2,\\
1 &\ \text{otherwise}.
\end{cases}
$$
Obviously, $\Psi$ is a positive measurable function on $\boldsymbol{ax+b}$. We show that it also satisfies
\begin{equation}\label{53}
\frac{\Psi(yz)} {\Psi(zy)}\le \omega_1(y)\omega_1(z)\quad(y,z\in\boldsymbol{ax+b}),
\end{equation}
where $\omega_1(a,b) = (1+a+|b|)$.
Let $y=(a,b),\,z=(c,d)\in \boldsymbol{ax+b}$. Then
$yz=(ac,ad+b)$ and $zy=(ac,bc+d)$. If $0<ac\le1$ or $ac\ge2$, then $\Psi(yz) = \Psi(zy) =1$ and hence (\ref{53}) holds trivially. Now assume $1<ac<2$. Then by the definition of $\Psi$ we have
\[
ac-1\le\Psi(zy)\le\Psi(zy)\omega_1(y)\omega_1(z)\quad\text{and}
\]
\begin{align*}
|ad+b|&=|a(bc+d)-b(ac-1)|\le a|bc+d|+|b|(ac-1)\\
     & \le \max\{ac-1,|bc+d|\}(a+|b|)=\Psi(zy)(a+|b|) \le \Psi(zy)\omega_1(y)\omega_1(z).
\end{align*}
Thus
$$
\Psi(yz)=\max\{ac-1,|ad+b|\}\le\Psi(zy)\omega_1(y)\omega_1(z).
$$
This shows that (\ref{53}) still holds if $1<ac<2$. Therefore, (\ref{53}) holds for all $y,z\in \boldsymbol{ax+b}$.

We now let $\psi=\ln \Psi$. Clearly, $\psi$ is a measurable function supported on $C_B$ and bounded on $B$. We show that it also satisfies the conditions 
\begin{equation}\label{52}
\underset{z\in\,C_B}{\mathrm{ess\,sup}}\,\frac{|\psi(z)|}{\omega_{\alpha}(z)}=\infty \quad\text{and}
\end{equation}
\begin{equation}\label{51}
|\psi(zy)-\psi(yz)|\le C\omega_{\alpha}(y)\omega_{\alpha}(z)\quad (y,z\in\boldsymbol{ax+b})
\end{equation}
for some constant $C>0$. 
 Indeed,
\begin{align*}
\underset{z\in\,C_B}{\mathrm{ess\,sup}}\,\frac{|\psi(z)|}{\omega_{\alpha}(z)}&\ge \sup\limits_{1<a<2}\,\frac{|\psi(a,a-1)|}{\omega_{\alpha}(a,a-1)}= \sup\limits_{1<a<2}\,\frac{|\ln(a-1)|}{(2a)^{\alpha}}=\infty.
\end{align*}
So (\ref{52}) is verified. To show (\ref{51}) we may assume, without loss of generality, that $\Psi(yz)\ge\Psi(zy)$. Then, using (\ref{53}), we obtain
\begin{align*}
\omega_{\alpha}(y)\omega_{\alpha}(z)&=(\omega_1(y)\omega_1(z))^{\alpha}\ge \left(\frac{\Psi(yz)}{\Psi(zy)}\right)^{\alpha}\ge \ln\left(\frac{\Psi(yz)}{\Psi(zy)}\right)^{\alpha}\\&=\alpha \left|\ln\,\frac{\Psi(yz)}{\Psi(zy)}\right|=\alpha|\psi(yz)-\psi(zy)|.
\end{align*}
It follows that $\psi$ satisfies (\ref{51}) with $C=1/\alpha$. Therefore, the function $\psi$ satisfies all the conditions of Theorem~\ref{conjugacygeneral}. This shows that $L^1(\boldsymbol{ax+b},\omega_{\alpha})$ is not weakly amenable. The proof is complete.

\end{proof}

We now equip $\boldsymbol{ax+b}$ with the discrete topology. It is readily seen that $H_{\bb} =\{(1,b):\, b\in\mathbb{R}\}$
is a normal subgroup of $\boldsymbol{ax+b}$, and $(\boldsymbol{ax+b})/H_{\bb}\cong(\mathbb{R}^+,\cdot)$ through the group homomorphism $[(a,b)]\mapsto a$.

\begin{proposition}\label{ax+b}
Let $\omega$ be a weight on $\boldsymbol{ax+b}$ that is bounded away from zero and is bounded on $H_{\bb}$. Then $\ell^1(\boldsymbol{ax+b},\omega)$ is weakly amenable if and only if $\omega$ is diagonally bounded on $\boldsymbol{ax+b}$.
\end{proposition}

\begin{proof}
The sufficiency is due to \cite[Proposition~4.1]{Shepelska}.

For the necessity, we assume that $\omega$ is not diagonally bounded. Let $\hat{\omega}$ be the function on $\boldsymbol{ax+b}$ defined by  $\hat{\omega}(z)=\inf\limits_{h\in H_{\bb}} \omega(zh)$.  Clearly, $\hat{\omega}$ is submultiplicative on $\boldsymbol{ax+b}$ and $\hat\omega(a,b)$ is independent of $b$. We simply denote $\hat\omega(a,b)$ by $\hat\omega(a)$. Then $\hat \omega$ is a submultiplicative function on $\bR^+$. It is easy to verify further that
\begin{equation}\label{omega hat}
\hat{\omega}(a) \leq \omega(a,b) \leq \tilde{c}\,\hat\omega(a) \quad ((a,b)\in \boldsymbol{ax+b}),
\end{equation}
where $\tilde{c}=\sup\limits_{h\in H_{\bb}} \omega(h)$. By our assumption $0<\tilde c < \infty$.

Consider the singleton set $B = \{(1,1)\}$. The conjugacy class of $B$ is
\[
C_B = \{y\cdot(1,1)\cdot y^{-1}: \, y\in \boldsymbol{ax+b}\} = \{(1,b): \, b>0\}.
\]

Define $\psi:\boldsymbol{ax+b}\to\mathbb{R}$ by
$$
\psi(a,b)=
\begin{cases}
\ln\left(\hat{\omega}(b)\hat{\omega}(b^{-1})\right) &\ \ \text{if}\ \ a=1,\,b>0,\\
0 &\ \ \text{otherwise}.
\end{cases}
$$
By definition, $\psi$ vanishes outside the conjugacy class $C_B$.
We show that
\begin{equation}\label{inequality 1}
|\psi(zy)-\psi(yz)|\le \omega(y)\omega(z)\quad (y,z\in\boldsymbol{ax+b}).
\end{equation}
Note that $zy$ and $yz$ always belong to the same conjugacy class for $y,z\in \boldsymbol{ax+b}$. So it suffices to verify (\ref{inequality 1}) for $zy,yz\in C_B$. Let $yz=(1,b)$, and $z=(k,l)$, $b,k>0$, $l\in\mathbb{R}$. Then
$$
y=(yz)z^{-1}=(k^{-1},(-l+bk)k^{-1}),\quad
zy=\left(1,bk\right).
$$
It follows that
\[
|\psi(zy)-\psi(yz)|=|\psi(1,bk)-\psi(1,b)|=\left|\ln\frac{\hat{\omega}(bk)\hat{\omega}((bk)^{-1})}{\hat{\omega}(b) \hat{\omega}((b)^{-1})}\right| \le \bigl|\ln\bigl(\hat{\omega}(k) \hat{\omega}(k^{-1})\bigr)\bigr|
\]
since
\[
 \frac{1}{\hat{\omega}(k)\hat{\omega}(k^{-1})}\le \frac{\hat{\omega}(bk)\hat{\omega}((bk)^{-1})}{\hat{\omega}(b)\hat{\omega}(b^{-1})}\le \hat{\omega}(k)\hat{\omega}(k^{-1}).
\]
But $\hat{\omega}(k)\hat{\omega}(k^{-1})\ge\hat{\omega}(e)\ge 1$. So $\bigl|\ln\bigl(\hat{\omega}(k) \hat{\omega}(k^{-1})\bigr)\bigr|=\ln\bigl(\hat{\omega}(k) \hat{\omega}(k^{-1})\bigr)\le \hat{\omega}(k) \hat{\omega}(k^{-1})$, which implies
$$
|\psi(zy)-\psi(yz)| \le \hat{\omega}(k)\hat{\omega}(k^{-1}).
$$
On the other hand, relation~(\ref{omega hat}) yields
\[
\omega(y) \geq \hat\omega(k^{-1}), \quad \omega(z) \geq \hat{\omega}(k).
\]
Thus we obtain (\ref{inequality 1}) as desired. Moreover, using (\ref{omega hat}) again, we have
\begin{align*}
\sup\limits_{x\in C_B}\,\frac{|\psi(x)|}{\omega(x)}&=\sup\limits_{b>0}\,\frac{|\psi(1,b)|}{\omega(1,b)}= \sup\limits_{b>0}\,\frac{|\ln(\hat{\omega}(b)\hat{\omega}(b^{-1}))|}{\omega(1,b)}\\
&\ge\frac{\sup\limits_{z\in\boldsymbol{ax+b}}|\ln(\omega(z)\omega(z^{-1}))|-|\ln\tilde{c}^2|}{\tilde c}=\infty,
\end{align*}
since $\omega$ is not diagonally bounded on $G$. From Corollary~\ref{conjugacy}, $\ell^1(\boldsymbol{ax+b},\omega)$ is not weakly amenable. The proof is complete.

\end{proof}

\section{Beurling algebra of quotient groups}\label{c5}

Let $G$ be a locally compact group, $\omega$ be a weight on $G$, and $H$ be a closed normal subgroup of $G$. Define $\hat\omega$ on the quotient group $G/H$ by
\[
\hat{\omega}([x])=\inf_{z\in[x]}\omega(z) = \inf_{\xi\in H}\omega(x\xi),
\]
where $[x]$ stands for the coset of $x$ in $G/H$ ($x\in G$). From \cite[Theorem~11.0]{HR} we know that $\hat\omega$ is a nonnegative upper semicontinuous and hence is a locally bounded measurable function on $G/H$. To avoid $\hat\omega$ being trivial, here and in the rest of this section we assume that $\omega$ is bounded away from zero. Then $\hat\omega$ is a locally bounded measurable weight function on $G/H$ \cite[Theorem~3.7.13]{RS}. As indicated in Section~\ref{Intro}, $\hat\omega$ is equivalent to a continuous weight. We note that in studying the weighted group algebra $\L1o$, requiring $\omega$ to be bounded away from zero is not really a restriction if $G$ is an amenable group. Indeed, if $G$ is amenable, then by \cite[Lemma~1]{White}\label{white} there exists a continuous positive character $\phi:G\to(\mathbb{R}^+,\cdot)$ such that $\phi\le\omega$ on $G$. Then $\tilde{\omega}=\omega/\phi\ge 1$ is a weight on $G$ and $L^1(G,\omega)$ is isometrically isomorphic to $L^1(G,\tilde{\omega})$ as a Banach algebra.

We are concerned with the relation between weak amenability of $\L1o$ and that of $L^1(G/H,\hat{\omega})$. First, as a simple consequence of Theorems~\ref{IN} and \ref{Thm_zhang} we obtain the following.

\begin{proposition}\label{quotient_IN}
Let $G$ be an IN group and H be a closed normal subgroup of $G$ such that $G/H$ is Abelian. Suppose that $\omega$ is a weight on~$G$ that is bounded away from zero. If $L^1(G,\omega)$ is weakly amenable, then so is $L^1(G/H,\hat{\omega})$.
\end{proposition}

\begin{proof}
If $L^1(G/H,\hat{\omega})$ were not weakly amenable, according to Theorem~\ref{Thm_zhang}, there would exist a continuous non-trivial group homomorphism $\Phi: G/H\to\mathbb{C}$ such that
$$
\sup_{[x]\in G/H}\,\frac{|\Phi([x])|}{\hat{\omega}([x])\hat{\omega}([x]^{-1})}<\infty.
$$
Then the natural extension $\tilde{\Phi}$ of $\Phi$ to $G$ defined by $\tilde{\Phi}(x)=\Phi([x])$ ($x\in G$) is a non-trivial continuous group homomorphism from $G$ to $\mathbb{C}$ and
$$
\sup_{x\in G}\,\frac{|\tilde{\Phi}(x)|}{\omega(x)\omega(x^{-1})}\le \sup\limits_{[x]\in G/H}\,\frac{|\Phi([x])|}{\hat{\omega}([x])\hat{\omega}([x]^{-1})}<\infty,
$$
since $\hat{\omega}([x])\le\omega(x)$ ($x\in G$). By Theorem~\ref{IN} this implies that $L^1(G,\omega)$ is not weakly amenable, contradicting our assumption.

\end{proof}

For the general case, according to the theory established in~\cite{RS}, there is a standard Banach algebra homomorphism $T$ from  $L^1(G,\omega)$ onto $L^1(G/H,\hat{\omega})$ defined by
\begin{equation}\label{operator T}
(Tf)([x])=\int\limits_{H}f(xh)\,dh \quad (f\in L^1(G,\omega),\,x\in G).
\end{equation}
The kernel of $T$ is a closed ideal in $\L1o$ and we denote it by $J_{\omega}(G,H)$. It was proved in \cite[Theorem~3.7.13]{RS} that $T$ induces an isometric isomorphism between $L^1(G,\omega)/J_{\omega}(G,H)$ and $L^1(G/H,\hat{\omega})$. So we are in the situation concerned by the following well-known result.

\begin{proposition}\label{WA_ideal}\cite[Proposition~2.4]{Gronbaek92}
Let $A$ be a weakly amenable Banach algebra and $I$ be a closed ideal in $A$. Then $A/I$ is weakly amenable if and only if $I$ has the trace extension property as described in the following.

For every $\lambda\in I^*$ satisfying $a\cdot\lambda=\lambda\cdot a$ ($a\in A$), there is a $\tau\in A^*$ such that $\tau|_I=\lambda$ and $\tau(ab)=\tau(ba)$ ($a,b\in A$).
\end{proposition}

We now investigate when $J_{\omega}(G,H)$ has the trace extension property as a closed ideal of $L^1(G,\omega)$. We start from proving that $J_{\omega}(G,H)$ is always complemented in $L^1(G,\omega)$ as a Banach subspace. For this we need two technical lemmas.

\begin{lemma}\cite[Proposition~8.1.16]{RS}\label{cover}
Let $H$ be a closed subgroup of a locally compact group $G$ and $U$ be a non-empty open set in $G$ with compact closure. Then there is a subset $Y$ of $G$ such that the family $\{UyH\}_{y\in Y}$ covers $G$ and is locally finite, i.e., every point of $G$ has a neighborhood intersecting at most finitely many members of the family.
\end{lemma}

The second lemma we need generalizes the investigation in \cite[Section~8.1]{RS} of the Bruhat function associated to a normal subgroup.

\begin{lemma}\label{bruhat}
Let $G$ be a locally compact group, $H$ be a closed normal subgroup of~$G$, and $\omega$ be a weight on $G$ bounded away from zero. Then there exists a continuous function $g\ge0$ on $G$ and a constant $c>0$ such that the following two conditions are satisfied:
\begin{equation}\label{gbruhat1}
\int\limits_{H} g(xh)\,dh=1\quad (x\in G)\quad\text{and}
\end{equation}
\begin{equation}\label{gbruhat2}
\int\limits_{H}g(xh)\omega(xh)\,dh\le c\,\hat{\omega}([x])\quad (x\in G).
\end{equation}
\end{lemma}

\begin{proof}
We first construct a continuous function $g_1$ on $G$ that satisfies
\begin{equation}\label{g1}
0<\int\limits_{H} g_1(xh)\,dh<\infty\quad (x\in G)\quad\text{and}
\end{equation}
\begin{equation}\label{g1_supp}
\text{supp}\,g_1\subset\{x\in G: \omega(x)\le c\,\hat{\omega}([x])\},
\end{equation}
where $c>0$ is a constant.

Consider a non-trivial non-negative function $f\in C_{00}(G)$. Let
\[
U=\{x\in G: f(x)>0\}.
\]
 Then $U\ne \emptyset$ is an open set with a compact closure. Let $\tilde c>0$ be a constant such that $\omega(u),\omega(u^{-1})\le\tilde{c}$ for every $u\in U$. (The existence of such $\tilde c$ is justified by the compactness of $\overline U$ and the continuity of $\omega$.) We set $c = 2\tilde c^2$.

By Lemma~\ref{cover}, there exists a set $Y\subset G$ such that the family $\{UyH\}_{y\in Y}$ covers~$G$ and is locally finite. For every $y\in Y$, by the definition of $\hat{\omega}$, there is $y_0\in[y]$ such that $\omega(y_0)\le2\hat{\omega}([y])$. We define $g_{1,y}(x)=f(xy_0^{-1})$ ($x\in G$). Clearly, $g_{1,y}\ge 0$ is a continuous function with compact support, and
\[
\{x: g_{1,y}(x) \ne 0\} = \{x: f(x) >0\}\cdot y_0 = Uy_0 \subset UyH.
\]
We now show that $g_{1,y}$ satisfies (\ref{g1_supp}), which is equivalent to
\begin{equation}\label{g1y_supp}
Uy_0\subset\{x\in G: \omega(x)\le c\,\hat{\omega}([x])\}.
\end{equation}
In fact, for each $u\in U$, by the choice of $y_0$ we have
\begin{align*}
\omega(uy_0)&\le\omega(u)\omega(y_0)\le 2\tilde{c}\,\hat{\omega}([y]) = 2\tilde{c}\inf_{h\in H}\omega(y_0h)\\
            &\leq 2\tilde{c}\omega(u^{-1})\inf_{h\in H}\omega(uy_0h) \leq 2\tilde{c}^2\hat{\omega}([uy_0]).
\end{align*}
So (\ref{g1y_supp}) holds.
Next we prove that $g_{1,y}$ satisfies
\begin{equation}\label{66}
0<\int\limits_{H} g_{1,y}(xh)\,dh<\infty\quad (x\in UyH).
\end{equation}
By definition, $g_{1,y}$ is a non-negative continuous function with a compact support. So the upper inequality holds.
Since $H$ is a normal subgroup of $G$, when $x\in UyH$ we have $xy_0^{-1} \in UH$, and hence there is $h_0\in H$ such that $xy_0^{-1}h_0 \in U$. Because $U$ is open, there is a non-trivial open subset $V$ of $H$ such that $xy_0^{-1}V \subset U$. Let $V_0 = y_0^{-1}Vy_0$. Then $V_0\ne \emptyset$ is an open subset of $H$ such that $xV_0y_0^{-1}\subset U$. Since $f>0$ on $U$, $g_{1,y}>0$ on $xV_0$. Therefore,
\[
\int\limits_{H} g_{1,y}(xh)\,dh \ge \int\limits_{V_0} g_{1,y}(xh)\,dh > 0.
\]

Now we let
$$
g_1=\sum\limits_{y\in Y} g_{1,y}
$$
Note that since $\{x: g_{1,y}(x)\ne0\}\subset UyH$ ($y\in Y$) and the family $\{UyH\}_{y\in Y}$ is locally finite, the sum in the definition of $g_1$ has only finitely many non-zero terms in a neighborhood of every point. This implies that $g_1$ is well-defined, and because each $g_{1,y}$ is continuous, $g_1$ is also continuous on $G$.
From (\ref{66}) and the local finiteness of $\{UyH\}_{y\in Y}$ it follows that (\ref{g1}) holds.
The inclusion (\ref{g1_supp}) also holds since it holds for each $g_{1,y}$. So the function $g_1$ satisfies all our requirements.

We then define the function $g$ by
$$
g(x)=\frac{g_1(x)}{\int\limits_H g_1(xh)\,dh}\quad (x\in G).
$$
Clearly, $g$ is a continuous non-negative function on $G$ and it satisfies
\[
\int\limits_H g(xh)\,dh=\int\limits_{H\ } \frac{g_1(xh)}{\int_H g_1(xht)\,dt}\ dh=\frac{\int_H g_1(xh)\,dh}{\int_H g_1(xt)\,dt}=1\quad (x\in G).
\]
So (\ref{gbruhat1}) is satisfied.
Moreover, it follows directly from (\ref{g1_supp}) and (\ref{gbruhat1}) that
$$
\int\limits_H g(xh)\omega(xh)\,dh\le c\,\hat{\omega}([x])\int\limits_H g(xh)\,dh=c\,\hat{\omega}([x]).
$$
So (\ref{gbruhat2}) is also satisfied.
The proof is complete.

\end{proof}

Let $g$ be a function ensured in Lemma~\ref{bruhat} and $T$ be the homomorphism given by (\ref{operator T}). Define
\beq\label{proj}
(Pf)(x)=(Tf)([x])\, g(x)\quad (x\in G,\ f\in L^1(G,\omega)).
\eeq
Then for each $f\in \L1o$, the function $P(f)$ is clearly measurable. By Weil's Formula and inequality~(\ref{gbruhat2}) we have
\begin{align*}
\int\limits_G |(Pf)(x)|\,\omega(x)\,dx &=\int\limits_{G/H}\int\limits_H |(Tf)([x])|g(xh)\omega(xh)\,dh\,d[x]\\
   &=\int\limits_{G/H}|(Tf)([x])|\int\limits_H g(xh)\omega(xh)\,dh\,d[x] \\
     &\le \int\limits_{G/H} |(Tf)([x])|\cdot c\,\hat{\omega}([x])\,d[x]=c\,\|Tf\|_{1,\hat{\omega}}\le c\,\|f\|_{1,\omega}.
\end{align*}
So $P: L^1(G,\omega)\to L^1(G,\omega)$ is a bounded operator with $\|P\|\le c$.

\begin{theorem}\label{complemented}
Let $G$ be a locally compact group, $H$ be a closed normal subgroup of~$G$, and $\omega$ be a weight on $G$ bounded away from zero. Then the mapping $P:L^1(G,\omega)\to L^1(G,\omega)$ defined by (\ref{proj}) is a continuous projection whose kernel is $J_{\omega}(G,H)$.
\end{theorem}

\begin{proof}
 Obviously, $\mathrm{ker}(P) = \mathrm{ker}(T) = J_{\omega}(G,H)$. So we only need to verify that $P^2 = P$. In fact,
\begin{align*}
(P^2f)(x)&=(P(Pf))(x)=(T(Pf))([x])\,g(x)=\left(\int\limits_H (Pf)(xh)\,dh\right) g(x)\\
         &=g(x)\int_H (Tf)([xh])g(xh)\,dh = g(x)(Tf)([x])\int\limits_H g(xh)\,dh\\
         &= (Tf)([x])\,g(x)=(Pf)(x)\quad (x\in G, \, f\in \L1o).
\end{align*}
Therefore, $P$ is a projection. The proof is complete.

\end{proof}

We do not know whether $J_\omega(G,H)$ has the trace extension property in general. The next lemma provides a sufficient condition for a complemented ideal to have the trace extension property.

\begin{lemma}\label{trace extension}
Let $A$ be a Banach algebra and $I$ be a closed complemented ideal in~$A$. Denote by $I_0$ the closure of
\[
\mathrm{lin}\{ at-ta:\, a\in A, t\in I\}.
\]
 Suppose that $A=I\oplus X$,where $X$ is a closed subspace of $A$ such that
$$
xy-yx\in I_0 \oplus X\quad (x,y\in X).
$$
Then $I$ has the trace extension property.
\end{lemma}

\begin{remark}
There are two important special cases for which conditions of Lemma~\ref{trace extension} are satisfied:

\setlength{\hangindent}{12pt}
\noindent
1. the complement $X$ of $I$ is a subalgebra of $A$;

\setlength{\hangindent}{12pt}
\noindent
 2. the complement $X$ is commutative, i.e., $xy=yx$ for all $x,y\in X$ (note that $xy$ may not be in $X$). In particular, this is the case if $A$ is Abelian.

\noindent Our Lemma~\ref{trace extension} generalizes \cite[Lemma~2.3]{LL}, where only the first case was concerned.
\end{remark}

\begin{proof}[Proof of Lemma~\ref{trace extension}.]
Let $\lambda\in I^*$ satisfy $\lambda\cdot a=a\cdot\lambda$ ($a\in A$). The condition really means $\lambda(ta)=\lambda(at)$ for all $t\in I$ and $a\in A$, or, equivalently, $\lambda|_{I_0} = 0$. Since $A=I\oplus X$, we have that $A^*=I^*\oplus X^*$. We show that $\tau=\lambda\oplus0$ is a trace extension of $\lambda$. Obviously, $\tau$ is a continuous linear functional on $A$, $\tau|_I=\lambda$, and $\tau|_{I_0 \oplus X} =0$. Now let $a,b\in A$ such that $a=t_1+x_1$ and $b=t_2+x_2$ with $t_1,t_2\in I$ and $x_1,x_2\in X$. We have $\lambda(t_1b) = \lambda(bt_1)$ and $\lambda(t_2x_1) = \lambda(x_1t_2)$. So
\[
\tau(ab)=\lambda(t_1b+x_1t_2)+\tau(x_1x_2)=\lambda(bt_1+t_2x_1)+\tau(x_1x_2) = \tau(ba) +\tau(x_1x_2-x_2x_1).
\]
Since $x_1x_2-x_2x_1 \in I_0 \oplus X$ by the assumption, $\tau(x_1x_2-x_2x_1)=0$. Therefore, $\tau(ab) = \tau(ba)$. This completes the proof.

\end{proof}

Combining Theorem~\ref{complemented} with Proposition~\ref{WA_ideal} and Lemma~\ref{trace extension}, we obtain the following.

\begin{proposition}\label{solv}
Let $G$ be a locally compact group, $H$ be a closed normal subgroup of $G$, and $\omega$ be a weight on $G$ bounded away from zero. Suppose that $X$ is a Banach space complement of $J_{\omega}(G,H)$ in $L^1(G,\omega)$ such that
$$
xy-yx\in J_0 \oplus X\quad (x,y\in X),
$$
where $J_0$ is the closure of $\mathrm{lin}\{f*j-j*f: \, f\in \L1o, j\in J_\omega(G,H)\}$.
Then weak amenability of $L^1(G,\omega)$ implies weak amenability of $L^1(G/H,\hat{\omega})$.
\end{proposition}

We now consider the special case when $G=G_1\times G_2$, $H=G_2$, and $\omega=\omega_1\times\omega_2$ with $\omega_i$ bounded away from zero on $G_i$ ($i=1,2$). In this case $G/H=G_1$,
$$
\hat{\omega}(x_1)=\omega_1(x_1)\inf_{x_2\in G_2}\omega_2(x_2)=\mathrm{const}\cdot\omega_1(x_1),
$$
and the operator $T:L^1(G,\omega)\to L^1(G/H,\hat{\omega})\cong L^1(G_1, \omega_1)$ is precisely given by
$$
T(f)(x_1)=\int\limits_{G_2}\, f(x_1,x_2)\,dx_2\quad (x_1\in G_1).
$$
Consider a non-negative function $h\in C_{00}(G_2)$ such that
$$
\int\limits_{G_2} h(x_2)\,dx_2=1.
$$
Then $g(x_1,x_2)=h(x_2)$ satisfies
$$
\int\limits_{G_2}{ g(x_1,x_2)}dx_2=\int\limits_{G_2}{h(x_2)}dx_2=1,\quad 
$$
\[
\int\limits_{G_2}{g(x_1,x_2)\omega(x_1,x_2)}dx_2=\omega_1(x_1)\int\limits_{G_2}{h(x_2)\omega_2(x_2)}dy =\textrm{const}\cdot\hat{\omega}(x_1)\quad (x_1\in G_1).
\]
Note that $L^1(G,\omega)=L^1(G_1,\omega_1)\hat{\otimes}L^1(G_2,\omega_2)$, and so we have
\begin{equation}\label{NX_tensor}
J_{\omega}(G,H)=L^1(G_1,\omega_1)\hat{\otimes} I_2,\quad X=L^1(G_1,\omega_1)\hat{\otimes}(\mathbb{C}h),
\end{equation}
where
\[
 I_2=\left\{f\in L^1(G_2,\omega_2): \int\limits_{G_2} f(x_2)\,dx_2=0\right\}.
\]

\begin{proposition}\label{tensor}
Let $G_1$, $G_2$ be locally compact groups and $\omega_i$ be a weight on $G_i$ bounded away from zero ($i=1,2$). Suppose that $L^1(G_1\times G_2,\omega_1\times\omega_2)$ is weakly amenable. Then both $L^1(G_1,\omega_1)$ and $L^1(G_2,\omega_2)$ are also weakly amenable.
\end{proposition}
\begin{proof}
Because of the symmetry, it is enough to show that $L^1(G_1,\omega_1)$ is weakly amenable. For this case, as has been discussed,
\[
L^1(G_1\times G_2,\omega_1\times\omega_2) = J_{\omega}(G,H) \oplus X
\]
with $J_\omega(G,H)$ and $X$ being given by (\ref{NX_tensor}).

For $f_1,f_2\in L^1(G_1,\omega_1)$ we have
\begin{align*}
&(f_1\otimes h)(f_2\otimes h)-(f_2\otimes h)(f_1\otimes h) = (f_1*f_2 - f_2*f_1)\otimes (h*h)\\
&= (f_1*f_2 - f_2*f_1)\otimes (h*h - h) + (f_1*f_2 - f_2*f_1)\otimes h.
\end{align*}
The second term of the last expression belongs to $X$. We show that the first term belongs to $J_0$. Denote $k = h*h - h$. It is easy to see that $k\in I_2$ and so $f_2\otimes k\in J_\omega(G,H)$. Let $(e_i)$ be a bounded approximate identity of $L^1(G_2,\omega_2)$. Then for each $i$
\[
(f_1\otimes e_i)(f_2\otimes k) - (f_2\otimes k)(f_1\otimes e_i)\in J_0,
\]
and hence
\begin{align*}
(f_1*f_2 - f_2*f_1)\otimes k &= \lim_i \bigl((f_1*f_2)\otimes (e_i*k) - (f_2*f_1)\otimes (k*e_i)\bigr)\\
                             &= \lim_i \bigl((f_1\otimes e_i)(f_2\otimes k) - (f_2\otimes k)(f_1\otimes e_i)\bigr) \in J_0.
\end{align*}
So we have shown that $(f_1\otimes h)(f_2\otimes h)-(f_2\otimes h)(f_1\otimes h) \in J_0 \oplus X$, and the condition of Proposition~\ref{solv} holds. Thus, $L^1(G_1,\omega_1) \cong \LH1o$ is weakly amenable if $L^1(G_1\times G_2,\omega_1\times\omega_2)$ is weakly amenable.

\end{proof}

\section{Beurling algebra of subgroups}\label{commu}
In spite of Proposition~\ref{tensor}, weak amenability of $L^1(G_1\times G_2, \omega)$ does not necessarily imply weak amenability of $L^1(G_1,\omega_1)$ even if the groups $G_1$, $G_2$ are commutative, where $\omega_1(x) =\omega(x,e_2)$ and $e_2$ is the unit of $G_2$. We give a counterexample in the following.

Let $G_1$, $G_2$ be Abelian locally compact groups and $G=G_1\times G_2$. Suppose that there exist continuous non-zero group homomorphisms $\Phi_i: G_i\to\mathbb{R}$ ($i=1,2$). For any $\alpha,\,\beta>0$ we define the function $\omega$ on $G$ as follows:
\begin{equation}\label{67}
\omega(x,y)=\left(1+|\Phi_1(x)|\right)^{\alpha}\left(1+|\Phi_1(x)+\Phi_2(y)|\right)^{\beta}\quad (x\in G_1,\,y\in G_2).
\end{equation}
It is readily seen that $\omega$ is a weight on $G$, and
$$
\omega_1(x)=\omega(x,e_2)=(1+|\Phi_1(x)|)^{\alpha+\beta} \quad (x\in G_1).
$$

\begin{example}\label{abelian_sum}
Let $G_1$, $G_2$, and $\omega$ be as above. If
$0<\alpha,\,\beta<1/2$ and $\alpha+\beta\ge 1/2$,
then $L^1(G,\omega)$ is weakly amenable, but $L^1(G_1,\omega_1)$ is not weakly amenable.
\end{example}
\begin{proof}

Since $\Phi_1$: $G_1 \to \mathbb R$ is a non-trivial continuous group homomorphism and
$$
\sup_{x\in\,G_1}\,\frac{|\Phi_1(x)|}{\omega_1(x)\omega_1(x^{-1})}= \sup_{x\in\,G_1}\,\frac{|\Phi_1(x)|}{(1+|\Phi_1(x)|)^{2(\alpha+\beta)}}<\infty
$$
if $\al+\be \ge 1/2$, $L^1(G_1,\omega_1)$ is not weakly amenable due to Theorem~\ref{Thm_zhang}. To show that $L^1(G,\omega)$ is weakly amenable, we consider any non-trivial continuous group homomorphism $\Phi: G\to\mathbb{R}$.  We have
$$
\sup_{g\in G}\,\frac{|\Phi(g)|}{\omega(g)\omega(g^{-1})}=\sup_{x\in G_1,y\in G_2}\frac{|\Phi(x,e_2)+\Phi(e_1,y)|}{\left(1+|\Phi_1(x)|\right)^{2\alpha}\left(1+ |\Phi_1(x)+\Phi_2(y)|\right)^{2\beta}}.
$$

Case 1. If there is $y\in G_2$ such that $\Phi(e_1, y)\ne 0$, then
$$
\sup_{g\in G}\,\frac{|\Phi(g)|}{\omega(g)\omega(g^{-1})}\ge \sup_{n\in \bN}\frac{|\Phi(e_1,y^n)|}{\omega(e_1,y^n)\omega(e_1,y^{-n})} = \sup_{n\in \bN}\frac{n\Phi(e_1,y)}{(1+ n |\Phi_2(y)|)^{2\be}} = \infty,
$$
since $\beta<1/2$.

Case 2. If $\Phi(e_1, y) = 0$ for all $y\in G_2$, then we can choose $x_0\in G_1$ such that $\Phi(x_0,e_2)\ne 0$. We can also choose $y\in G_2$ such that $\Phi_2(y) \neq 0$. For each $x\in G_1$, we take an $n=n(x)\in \bN$ such that
\[
\left|n+\frac{\Phi_1(x)}{\Phi_2(y)}\right|\le 1.
\]
It then follows that
\[
\left|\Phi_1(x)+\Phi_2\left(y^{n}\right)\right|=|\Phi_1(x)+n\Phi_2(y)|\le |\Phi_2(y)|.
\]
Hence,
\begin{align*}
\sup_{g\in G}\,\frac{|\Phi(g)|}{\omega(g)\omega(g^{-1})}
&\ge\ \sup\limits_{x\in\,G_1}\,\frac{|\Phi(x,e_2)|} {\left(1+|\Phi_1(x)|\right)^{2\alpha}(1+|\Phi_1(x)+\Phi_2(y^{n})|)^{2\beta}}
\\
&\ge\ \sup\limits_{x\in\,G_1}\,\frac{|\Phi(x,e_2)|} {\left(1+|\Phi_1(x)|\right)^{2\alpha}(1+|\Phi_2(y)|)^{2\beta}}
\\
&\ge \sup\limits_{m\in \bN}\,\frac{m|\Phi(x_0,e_2)|} {\left(1+m|\Phi_1(x_0)|\right)^{2\alpha}(1+|\Phi_2(y)|)^{2\beta}} =\infty,
\end{align*}
because $\alpha<1/2$.

So, we have shown that
$$
\sup_{g\in G}\,\frac{|\Phi(g)|}{\omega(g)\omega(g^{-1})}=\infty
$$
for every non-trivial continuous group homomorphism $\Phi: G\to\mathbb{R}$. Therefore, $\L1o$ is weakly amenable by Theorem~\ref{Thm_zhang} (see \cite[Theorem~3.5]{zhang}).

\end{proof}

Example~\ref{abelian_sum} also shows that, unlike the group algebra case, in general weak amenability of a Beurling algebra on an Abelian group $G$ does not imply weak amenability of the induced Beurling algebra on a subgroup of $G$.
However, the implication is true for certain ``large'' open subgroups. We first give a technical lemma dealing with extension of a group homomorphism.

\begin{lemma}\label{extension}
Let $G$ be a locally compact Abelian group and $H$ be an open subgroup of $G$. Then any continuous group homomorphism $\Phi: H\to \mathbb{C}$ can be extended to a continuous group homomorphism $\tilde{\Phi}: G\to \mathbb{C}$.
\end{lemma}

\begin{proof}
By Zorn's Lemma, it suffices to show that for every $g\in G$ we can extend $\Phi$ to the open subgroup $H_g=\bigcup\limits_{n\in\mathbb{Z}} g^nH=\{g^nh: h\in H,\,n\in\mathbb{Z}\}$ of $G$.

Suppose first that there exists $m\in\mathbb{N}$ such that $g^m\in H$. Let $m_0$ be the smallest such number. Then we denote $\alpha=\frac 1 {m_0} \Phi(g^{m_0})$ and define $\tilde{\Phi}(g^nh)=n\alpha+\Phi(h)$ ($h\in H,\,n\in\mathbb{Z}$). It is easy to see that $\tilde{\Phi}$ is a group homomorphism on $H_g$. In fact, the only non-trivial assertion one needs to verify is that the extension is well-defined, i.e., if $g^{n_1}h_1=g^{n_2}h_2$ then $n_1\alpha+\Phi(h_1)=n_2\alpha+\Phi(h_2)$. But in this case $g^{n_1-n_2}=h_2h_1^{-1}\in H$, and so $n_1-n_2=km_0$ for some $k\in\mathbb{Z}$. Because $\Phi$ is a group homomorphism on $H$, we then have
$$
\Phi(h_2)-\Phi(h_1)=\Phi(h_2h_1^{-1})=\Phi(g^{n_1-n_2})=k\Phi(g^{m_0})=km_0\alpha=(n_1-n_2)\alpha,
$$
which implies the desired equality $n_1\alpha+\Phi(h_1)=n_2\alpha+\Phi(h_2)$. The extension $\tilde{\Phi}$ is also continuous on $H_g$. Indeed, let $\{t_{\gamma}=g^{n_\gamma}h_{\gamma}\}_{\gamma\in\Gamma}\subset H_g$ be a net that converges to some $t=g^nh\in H_g$. We have $g^{n_{\gamma}-n}h_{\gamma}\overset{\gamma}{\to} h$. Since $H$ is open, there is $\gamma_0\in\Gamma$ such that $g^{n_{\gamma}-n}\in H$ for $\gamma\ge\gamma_0$. Then from the continuity of $\Phi$ on $H$ it follows that
$$
\tilde{\Phi}(g^{n_{\gamma}-n}h_{\gamma})={\Phi}(g^{n_{\gamma}-n}h_{\gamma})\overset{\gamma}{\to}\Phi(h)=\tilde{\Phi}(h).
$$
Using the fact that $\tilde{\Phi}$ is a group homomorphism, we finally obtain
$$
\tilde{\Phi}(t_{\gamma})=\tilde{\Phi}(g^{n_{\gamma}}h_{\gamma})=\tilde{\Phi}(g^{n_{\gamma}-n}h_{\gamma})+ \tilde{\Phi}(g^n)\overset{\gamma}{\to}\tilde{\Phi}(h)+ \tilde{\Phi}(g^n)=\tilde{\Phi}(g^nh)=\tilde{\Phi}(t).
$$

Now assume that $g^n\notin H$ for all $n\in\mathbb{N}$. Then we put $\tilde{\Phi}(g^nh)=\Phi(h)$ ($h\in H$, $n\in\mathbb{Z}$). Obviously, $\tilde{\Phi}$ is a group homomorphism on $H_g$. We now show that it is continuous. Let $g^{n_\gamma}h_{\gamma}\overset{\gamma}{\to} g^nh$ ($n_{\gamma},n\in\mathbb{Z}$, $h_{\gamma},h\in H$). Then, as above, $g^{n_{\gamma}-n}h_{\gamma}\overset{\gamma}{\to} h$ and, because $H$ is open, there is $\gamma_0$ such that $g^{n_{\gamma}-n}\in H$ for $\gamma\ge\gamma_0$. But our assumption on $g$ implies that this is possible only when $n_{\gamma}=n$ for $\gamma\ge\gamma_0$, and so $h_{\gamma}\overset{\gamma}{\to} h$. Therefore,
\[
\tilde{\Phi}(g^{n_{\gamma}}h_{\gamma})=\Phi(h_{\gamma})\overset{\gamma}{\to}\Phi(h)=\tilde{\Phi}(g^nh).
\]
 This shows that $\tilde{\Phi}$ is continuous. The proof is complete.

\end{proof}

In general, one cannot expect that a group homomorphism  $\Phi$ from a normal subgroup $H$ of $G$ has an extension to the whole $G$. In fact, if such extension exists then $\Phi$ must satisfy $\Phi(ghg^{-1}) = \Phi(h)$ for all $g\in G$ and $h\in H$. It turns out that the latter condition is also sufficient for semidirect product group $G = L\ltimes H$, where $H$ is a normal subgroup and $L$ is a subgroup of $G$ such that $L\cap H = \{e\}$.

\begin{proposition}
Let $G=L\ltimes H$ and $\Phi:H\to\mathbb{R}$ be a group homomorphism. Then $\Phi$ extends to a group homomorphism $\widetilde{\Phi}:G\to\mathbb{R}$ if and only if
\beq\label{semidirect}
\Phi(lhl^{-1})=\Phi(h)  \quad (l\in L, h\in H).
\eeq
 Moreover, if $H$ is open in $G$ then $\widetilde{\Phi}$ is continuous whenever $\Phi$ is continuous.
\end{proposition}

\begin{proof}
The necessity part is trivial.

For sufficiency, we note that every $g\in G$ may be uniquely expressed in the form $g=lh$. 
 Suppose that (\ref{semidirect}) holds. We then extend $\Phi$ to $\widetilde{\Phi}$ on the whole $G$ simply by letting 
 $\widetilde\Phi(g) = \Phi(h)$ ($g = lh $, $l\in L$, $h\in H$). 
 It is a group homomorphism because for any $g_1=l_1h_1,\, g_2=l_2h_2\in G$ we have
\begin{align*}
\widetilde{\Phi}(g_1g_2)&=\widetilde{\Phi}(l_1h_1l_2h_2)= \widetilde{\Phi}\left((l_1l_2)(l_2^{-1}h_1l_2h_2)\right)=\Phi(l_2^{-1}h_1l_2h_2)\\
                        &=\Phi(l_2^{-1}h_1l_2)+\Phi(h_2)=\Phi(h_1)+ \Phi(h_2)=\widetilde{\Phi}(g_1)+\widetilde{\Phi}(g_2).
\end{align*}

Assume now that $H$ is open in $G$ and that $\Phi$ is continuous on $H$.  Let $g_i=l_ih_i\to g=lh$ ($h_i,h\in H$, $l_i,l\in L$). Then $l^{-1}l_ih_i\to h\in H$. Since $H$ is open, it follows that $l^{-1}l_ih_i\in H$ ($i\ge i_0$) for some $i_0$. Then $l^{-1}l_i\in H\cap L$ and hence $l_i = l$ for $i\ge i_0$. This implies that $h_i\to h$. Using the continuity of $\Phi$ we finally obtain
$$
\widetilde{\Phi}(g_i)=\Phi(h_i)\to\Phi(h)= \widetilde{\Phi}(g).
$$
Therefore, $\widetilde{\Phi}$ is also continuous.
\end{proof}


\begin{proposition}\label{compact quotient}
Let $G$ be a locally compact IN group and $\omega$ be a weight on it. Suppose that $H$ is a commutative subgroup of $G$, and suppose that every continuous group homomorphism $\Phi:H\to\bC$ can be extended to the whole $G$. If there is $c>0$ such that for each $x\in G$ there is $k=k(x)\in \bN$ for which $x^k \in H$ and
\begin{equation}\label{grow control}
\frac{\omega(x^k)\omega(x^{-k})}{k} \leq c\, \omega(x)\omega(x^{-1}),
\end{equation}
then weak amenability of $L^1(G,\omega)$ implies weak amenability of $L^1(H,\omega|_H)$.
\end{proposition}

\begin{remark}
In particular, the conditions of Proposition~\ref{compact quotient} are satisfied when $G/H$ is a torsion group (see \cite[A.1]{HR}).
\end{remark}

\begin{proof}[Proof of Proposition~\ref{compact quotient}.]
If $L^1(H,\omega|_H)$ is not weakly amenable, by Theorem~\ref{Thm_zhang} there is a non-trivial continuous group homomorphism $\Phi:H\to\mathbb{C}$ such that
$$
\sup\limits_{h\in\,H}\,\frac{|\Phi(h)|}{\omega(h)\omega(h^{-1})} = r<\infty.
$$
By our assumption, $\Phi$ can be extended to a continuous group homomorphism $\widetilde{\Phi}:G\to\mathbb{R}$. We have
\[
\frac{|\widetilde{\Phi}(x)|}{\omega(x)\omega(x^{-1})} = \frac{|\Phi(x^k)|}{\omega(x^k)\omega(x^{-k})} \, \frac{\omega(x^k)\omega(x^{-k})}{k} \, \frac{1}{\omega(x)\omega(x^{-1})}\, \leq \, rc
\]
since $x^k\in H$, where $k = k(x)\in\mathbb{N}$ is such that (\ref{grow control}) is satisfied. Then, by Theorem~\ref{IN}, $\L1o$ is not weakly amenable.

\end{proof}

\begin{corollary}
Let $G$ be a locally compact [IN] group and $H$ be a commutative subgroup of $G$ of finite index. Suppose that each continuous group homomorphism from $H$ to $\bC$ can be continuously extended to the whole $G$. Then, for every weight $\omega$ on $G$ such that $\L1o$ is weakly amenable, $L^1(H,\omega|_H)$ is also weakly amenable.
\end{corollary}
\begin{proof}
Suppose, to the contrary, that $L^1(H,\omega|_H)$ is not weakly amenable. Then, since $H$ is commutative, Theorem~\ref{Thm_zhang} implies the existence of a non-trivial continuous group homomorphism $\Phi: H\to\bC$ and a constant $c>0$ such that
$$
\frac{|\Phi(h)|}{\omega(h)\omega(h^{-1})}\le c\quad (h\in H).
$$
By the assumption $\Phi$ extends to a continuous group homomorphism $\widetilde{\Phi}:G\to \bC$.
Because $H$ is of finite index, there exist $g_1, g_2,\ldots,g_n\in G$ such that $G=\cup_{i=1}^n\, g_iH$. Hence, every $g\in G$ can be written in the form $g=g_ih$ for some $1\le i\le n$, $h\in H$, and so
\begin{align*}
\frac{\left|\widetilde{\Phi}(g)\right|}{\omega(g)\omega(g^{-1})}&\le \frac{\left|\widetilde{\Phi}(h)\right|+\left|\widetilde{\Phi}(g_i)\right|}{\omega(g_ih) \omega(h^{-1}g_i^{-1})}\le \frac{\left|\widetilde{\Phi}(h)\right|+\left|\widetilde{\Phi}(g_i)\right|}{\omega(h) \omega(h^{-1})}\cdot \omega(g_i) \omega(g_i^{-1})\\ &\le \max\limits_{1\le i\le n}\left(c+{\left|\widetilde{\Phi}(g_i)\right|}\right)\omega(g_i) \omega(g_i^{-1})=const.
\end{align*}
It follows that $\L1o$ is not weakly amenable by Theorem~\ref{IN}, which contradicts our assumption.

\end{proof}

Given a locally compact group $G$ and a closed normal subgroup $H$ of it, we have seen that weak amenability of $\L1o$  does not pass to $L^1(H,\omega|_H)$ in general even $G$ is commutative. One may wonder whether the condition that both $L^1(H,\omega|_H)$ and $L^1(G/H,\hat{\omega})$ are weakly amenable forces $\L1o$ to be weakly amenable.
It turns out that the answer is also negative. A counterexample is as follows.

\begin{example}\label{counterexample} We consider $G = \boldsymbol{ax+b}$ and  $H = H_{\bb}$. Suppose that $w$ is a weight on $(\mathbb{R}^+,\cdot)$ that is not diagonally bounded, but such that $\ell^1(\mathbb{R}^+,w)$ is weakly amenable. (For example, we can take $w(a)=(1+|\ln a|)^{\alpha}$, $0<\alpha<1/2$.)
 We then define $\omega$ on $\boldsymbol{ax+b}$ by $\omega(a,b)=w(a)$ ($a>0$). Clearly, $\omega$ is a weight on $G$, $\omega|_H = \mathrm{constant}$, and $\hat\omega = w$. So $\ell^1(H,\omega|_H)$ and $\ell^1\left((\boldsymbol{ax+b})/H,\hat{\omega}\right)$ are both weakly amenable. But by our assumption $\omega$ is not diagonally bounded, and so $\ell^1(\boldsymbol{ax+b},\omega)$ is not weakly amenable due to Proposition~{\ref{ax+b}}.
\end{example}
Even $G$ is finitely generated, this situation could happen.
\begin{example}
Let $\mathbb{Z}\left[\frac12\right]$ denote the set of all dyadic fractions, i.e., the set of all rational numbers whose binary expansion is finite. Consider the countable subgroup $G_2$ of $\boldsymbol{ax+b}$ defined by
$$
G_2=\left\{(2^n,b): n\in\mathbb{Z},\,b\in \mathbb{Z}\left[\frac12\right]\right\}.
$$
In fact, $G_2$ is the subgroup of $\boldsymbol{ax+b}$ generated by the elements $(2,0)$ and $(1,1)$, and so it is a finitely generated amenable group. Let
$$
H_2=H_{\bb}\cap G_2=\left\{(1,b): b\in \mathbb{Z}\left[\frac12\right]\right\}.
$$
Then $H_2$ is a normal subgroup of $G_2$ and $G_2/H_2\cong(\mathbb{Z},+)$. On $G_2$ we consider the weight $\omega_{\alpha}$ ($0<\alpha<1/2$) defined by
$$
\omega_{\alpha}(2^n,b)=(1+|n|)^{\alpha}\quad (n\in\mathbb{Z}).
$$
The same argument as in Example~\ref{counterexample} shows that $\ell^1(G_2,\omega_{\alpha})$ is not weakly amenable while both $\ell^1(H_2,\omega_{\alpha})$, which is isomorphic to $\ell^1(H_2)$, and $\ell^1(G_2/H_2,\hat{\omega}_{\alpha})$, which is isometrically isomorphic to $\ell^1(\mathbb{Z},\omega_{\alpha})$, are weakly amenable. We are grateful to N.~Spronk for this observation.
\end{example}






\bibliographystyle{plain}


\end{document}